\documentclass[reqno]{amsart}
\usepackage{amsmath,amsthm,amssymb,bbm,color,verbatim,tikz,mathabx,mathrsfs}

\usetikzlibrary{matrix}

\input xy 
\xyoption{all}

\numberwithin{equation}{section}

\usepackage{hyperref}
\hypersetup{colorlinks=true, linkcolor=blue, citecolor=blue}

\newcommand{\ZFC}{{\rm ZFC}}

\renewcommand{\P}{{\mathbb P}}

\newcommand{\Q}{{\mathbb Q}}
\newcommand{\R}{{\mathbb R}}
\renewcommand{\H}{{\mathbb H}}

\newcommand{\forces}{\Vdash}
\newcommand{\forced}{\Vdash}

\newcommand{\restrict}{\upharpoonright}

\newcommand{\<}{\langle}
\renewcommand{\>}{\rangle}

\newcommand{\st}{:}

\newcommand{\NS}{{\mathop{\rm NS}}}

\newcommand{\Tr}{{\mathop{\rm Tr}}}
\newcommand{\wc}{{\mathop{\rm wc}}}

\newcommand{\Refl}{{\mathop{\rm Refl}}}

\renewcommand{\and}{\mathop{\&}}

\newtheorem{theorem}{Theorem}
\newtheorem{lemma}[theorem]{Lemma}

\newtheorem*{theorem3}{Theorem 3}

\theoremstyle{definition}
\newtheorem{definition}{Definition}
\newtheorem{question}{Question}
\newtheorem{remark}{Remark}

\date{\today}

\begin{document}

\title[The weakly compact reflection principle]{The weakly compact reflection principle need not imply a high order of weak compactness}

\author[Brent Cody]{Brent Cody}
\address[Brent Cody]{ 
Virginia Commonwealth University,
Department of Mathematics and Applied Mathematics,
1015 Floyd Avenue, PO Box 842014, Richmond, Virginia 23284, United States
} 
\email[B. ~Cody]{bmcody@vcu.edu} 
\urladdr{http://www.people.vcu.edu/~bmcody/}

\author[Hiroshi Sakai]{Hiroshi Sakai}
\address[Hiroshi Sakai]{ 
Kobe University,
Graduate School of System Informatics,
1-1 Rokkodai, Nada, Kobe 657-8501, Japan
} 
\email[H. ~Sakai]{hsakai@people.kobe-u.ac.jp} 
\urladdr{http://www2.kobe-u.ac.jp/~hsakai/}

\begin{abstract}
The \emph{weakly compact reflection principle} $\Refl_{\wc}(\kappa)$ states that $\kappa$ is a weakly compact cardinal and every weakly compact subset of $\kappa$ has a weakly compact proper initial segment. The weakly compact reflection principle at $\kappa$ implies that $\kappa$ is an $\omega$-weakly compact cardinal. In this article we show that the weakly compact reflection principle does not imply that $\kappa$ is $(\omega+1)$-weakly compact. Moreover, we show that if the weakly compact reflection principle holds at $\kappa$ then there is a forcing extension preserving this in which $\kappa$ is the least $\omega$-weakly compact cardinal. Along the way we generalize the well-known result which states that if $\kappa$ is a regular cardinal then in any forcing extension by $\kappa$-c.c. forcing the nonstationary ideal equals the ideal generated by the ground model nonstationary ideal; our generalization states that if $\kappa$ is a weakly compact cardinal then after forcing with a `typical' Easton-support iteration of length $\kappa$ the weakly compact ideal equals the ideal generated by the ground model weakly compact ideal.
\end{abstract}



\maketitle


\section{Introduction}


For a regular cardinal $\kappa$, the \emph{stationary reflection principle} $\Refl(\kappa)$ states that every stationary subset of $\kappa$ has a stationary proper initial segment. If we modify the statement of $\Refl(\kappa)$ to reference ideals other than the nonstationary ideals on cardinals $\gamma\leq\kappa$ we obtain new reflection principles. The \emph{$\Pi^1_n$-reflection principle} $\Refl_{n}(\kappa)$, independently defined by the authors of this paper, states that $\kappa$ is a $\Pi^1_n$-indescribable cardinal and every $\Pi^1_n$-indescribable subset of $\kappa$ has a $\Pi^1_n$-indescribable proper initial segment. These reflection principles $\Refl_n(\kappa)$ can be seen to generalize a certain type of stationary reflection principle as follows. Since a set $S\subseteq\gamma$ is $\Pi^1_0$-indescribable if and only if $\gamma$ is inaccessible and $S$ is stationary \cite{MR2252250}, it follows that $\Refl_0(\kappa)$ holds if and only if $\kappa$ is inaccessible and for every stationary $S\subseteq\kappa$ there is an inaccessible $\gamma<\kappa$ such that $S\cap\gamma$ is stationary in $\gamma$. Thus, $\Refl_n(\kappa)$ is a generalization of a natural stationary reflection principle. Since a set $S\subseteq\gamma$ being $\Pi^1_n$-indescribable is expressible by a $\Pi^1_{n+1}$-sentence over $(V_\gamma,\in,S)$, it follows that if $\kappa$ is a $\Pi^1_{n+1}$-indescribable cardinal then $\Refl_n(\kappa)$ holds. 

The second author defined the $\Pi^1_n$-reflection principles in order to generalize Jensen's characterization of weak compactness in $L$. Recall that Jensen proved \cite{MR0309729} that in $L$, a cardinal $\kappa$ is weakly compact if and only if the stationary reflection principle holds at $\kappa$. In \cite{MR3416912}, the second author and Bagaria-Magidor, showed that Jensen's characterization of weak compactness in $L$ can be generalized: in $L$ a cardinal $\kappa$ is $\Pi^1_{n+1}$-indescribable if and only if the $\Pi^1_n$-reflection principle holds at $\kappa$.\footnote{The second author proved that in $L$, a cardinal $\kappa$ is $\Pi^1_{n+1}$-indescribable if and only if the $\Pi^1_n$-reflection principle holds. Independently, Bagaria-Magidor showed that in $L$, $\kappa$ is a $\Pi^1_{n+1}$-indescribable cardinal if and only if $\kappa$ is what they call ``$n$-stationary''; this is the version appearing in \cite{MR3416912}. Thus, in $L$, a cardinal $\kappa$ is $\Pi^1_{n+1}$-indescribable if and only if $\kappa$ is $n$-stationary if and only if the $\Pi^1_n$-reflection principle holds.}

It follows from Jensen's characterization of weak compactness in $L$ mentioned above, that a cardinal $\kappa$ being greatly Mahlo need not imply that the stationary reflection principle $\Refl(\kappa)$ holds; indeed in $L$, the stationary reflection principle $\Refl(\kappa)$ fails if $\kappa$ is the least greatly Mahlo cardinal. Using standard methods one can prove this consistency result using a forcing which adds a non-reflection stationary set: if $\kappa$ is greatly Mahlo then there is a forcing extension in which $\Refl(\kappa)$ fails and $\kappa$ remains greatly Mahlo.\footnote{The first author would like to thank James Cummings for pointing this out.} The notions of great Mahloness and order of a stationary set can be generalized to yield notions of great $\Pi^1_n$-indescribability and orders of $\Pi^1_n$-indescribable sets. See Definition \ref{definition_alpha_weakly_compact} and the surrounding discussion for a review of the relevant material on orders of $\Pi^1_1$-indescribability; for more on orders of $\Pi^1_n$-indescribability consult \cite{Cody:AddingANonReflectingWeaklyCompactSet} and \cite{MR2252250}. The theorem of Bagaria-Magidor-Sakai mentioned above shows that the fact that $\Refl(\kappa)$ does not hold at the least greatly Mahlo cardinal can be generalized to $\Pi^1_n$-indescribability: in $L$, the $\Pi^1_n$-reflection principle $\Refl_n(\kappa)$ fails if $\kappa$ is the least greatly $\Pi^1_n$-indescribable cardinal. Thus $\kappa$ being greatly $\Pi^1_n$-indescribable need not imply the $\Pi^1_n$-reflection principle. The first author showed that in the case where $n=1$, the forcing that adds a non-reflection stationary set can be generalized to show that if $\kappa$ is $\xi$-$\Pi^1_1$-indescribable where $\xi<\kappa^+$ is some fixed ordinal then there is a forcing which adds a \emph{non-reflecting $\Pi^1_1$-indescribable subset of $\kappa$},\footnote{A set $S\subseteq\kappa$ is called a \emph{non-reflecting $\Pi^1_n$-indescribable subset of $\kappa$} if $S$ is a $\Pi^1_n$-indescribable subset of $\kappa$ and for every $\gamma<\kappa$ the set $S\cap \gamma$ is not a $\Pi^1_n$-indescribable subset of $\gamma$. See \cite{Cody:AddingANonReflectingWeaklyCompactSet} for more on non-reflecting $\Pi^1_n$-indescribable sets.} thus killing $\Refl_1(\kappa)$, and preserves the $\xi$-$\Pi^1_1$-indescribability of $\kappa$. \footnote{It is not known whether or not there is a forcing which adds a non-reflecting $\Pi^1_1$-indescribable subset to $\kappa$ while preserving the great $\Pi^1_1$-indescribability of $\kappa$. See Section \ref{section_questions}.}


Mekler and Shelah also observed that $\Refl_0(\kappa)$ implies that $\kappa$ is $\omega$-Mahlo and proved \cite{MR1029909} that $\Refl(\kappa)$ need not imply that $\kappa$ is $(\omega+1)$-Mahlo\footnote{This also follows from a result of Magidor \cite{MR683153} which states that $\Refl(\aleph_{\omega+1})$ is consistent relative to a supercompact cardinal. We emphasize Mekler and Shelah's proof because our proof will be a generalization of theirs.} by showing that if $\Refl(\kappa)$ holds then there is a forcing extension in which $\Refl(\kappa)$ holds at the least $\omega$-Mahlo cardinal. The first author observed that the $\Pi^1_n$-reflection principle $\Refl_n(\kappa)$ implies that $\kappa$ is $\omega$-$\Pi^1_n$-indescribable and asked \cite{Cody:AddingANonReflectingWeaklyCompactSet}: can the $\Pi^1_n$-reflection principle hold at the least $\omega$-$\Pi^1_n$-indecribable? In this article we answer this question affirmatively in the case where $n=1$. We refer to $\Refl_1(\kappa)$ as the weakly compact reflection principle and write $\Refl_{\wc}(\kappa)$ instead of $\Refl_1(\kappa)$.\footnote{It is not clear whether the methods in this article will provide an answer to this question for $n>1$. For a more detailed discussion of this and other questions, see Section \ref{section_questions} below.}


\begin{theorem}\label{theorem_main_theorem} 
Suppose the weakly compact reflection principle $\Refl_\wc(\kappa)$ holds. Then there is a forcing extension in which $\Refl_\wc(\kappa)$ holds and $\kappa$ is the least $\omega$-weakly compact cardinal.
\end{theorem}

To prove Theorem \ref{theorem_main_theorem} we will define an Easton-support forcing iteration of length $\kappa$ which kills the $\omega$-weak compactness of all cardinals less than $\kappa$, preserves $\Refl_{\wc}(\kappa)$ and thus preserves the $\omega$-weak compactness of $\kappa$. Before summarizing our proof of Theorem \ref{theorem_main_theorem}, let us first recall a few basic concepts. The $\omega$-weak compactness of a cardinal $\gamma\leq\kappa$ is witnessed by the fact that all sets in a particular $\omega$-sequence $\vec{S}=\<S^\gamma_n\st n<\omega\>$ are weakly compact, i.e. positive with respect to the weakly compact ideal at $\gamma$ (see Lemma \ref{lemma_characterization_of_omega_weak_compactness}). Sun proved \cite{MR1245524} that if $\gamma$ is a weakly compact cardinal then a set $E\subseteq\gamma$ is weakly compact if and only if $E\cap C\neq\emptyset$ for every $1$-club $C\subseteq\gamma$; we review the definition of $1$-club and relevant results below: see Definition \ref{definition_1_closed_1_club} and Lemma \ref{lemma_1_club_characterization}. Hellsten showed that the weak compactness of a set $S\subseteq\gamma$, whose complement $\gamma\setminus S$ also happens to be weakly compact, can be killed by a forcing which shoots a $1$-club through $\gamma\setminus S$ and preserves the weak compactness of $\gamma$; we provide a proof of this result, see Lemma \ref{theorem_1_club_shooting}, which differs slightly from Hellsten's argument in that it does not use the indescribability characterization of weak compactness. We could easily define a forcing iteration that kills all $\omega$-weakly compact cardinals $\gamma<\kappa$, simply by shooting a $1$-club through $\gamma\setminus S^\gamma_n$ for some fixed $n<\omega$ and for all weakly compact $\gamma<\kappa$. However, this would kill all $(n+1)$-weakly compact cardinals below $\kappa$ and thus would not preserve the $\omega$-weak compactness of $\kappa$. Our iteration will use a lottery sum at stage $\gamma$ to generically select what degrees of weak compactness to kill at that stage. For technical reasons explained below in Remark \ref{remark_intuitive_comment}, to kill the $\omega$-weak compactness of some $\gamma<\kappa$ our iteration will kill the weak compactness of all sets in some final segment $\vec{S}\restrict[m,\omega)$ of $\vec{S}$ by shooting $1$-clubs through $\gamma\setminus S^\gamma_n$ for \emph{all} $n\in[m,\omega)$; this is a key difference between our forcing and the forcing used by Mekler and Shelah (see \cite[Theorem 9]{MR1029909}). The preservation of $\Refl_{\wc}(\kappa)$ will be proved by lifting elementary embeddings witnessing the weak compactness of various sets and then applying the following lemma, which is due to Gitik and Shelah \cite[Lemma 1.13]{MR1780068}, and was employed by Mekler and Shelah \cite[Lemma 6]{MR1029909}.
\begin{lemma}[Gitik-Shelah, \cite{MR1780068}]\label{lemma_gitik_shelah}
Suppose $\lambda$ is a regular cardinal and $\Q$ is a notion of forcing which is $\lambda$-c.c. Suppose $I$ is a normal $\lambda$-complete ideal on $\lambda$, $S\in I^+$ and $\vec{q}=\<q_\alpha\st\alpha\in S\>$ is a sequence of conditions. Then there is a set in the dual filter $C\in I^*$ so that for all $\alpha\in C\cap S$, 
\[q_\alpha\forces \text{``$\{\beta\in S\st q_\beta\in\dot{G}\}$ is positive with respect to the ideal generated by $I$''.}\]
\end{lemma}
When applying Lemma \ref{lemma_gitik_shelah} to our iteration, we will need to know that after forcing with our iteration, the ideal generated by the ground model weakly compact ideal is equal to the weakly compact ideal of the extension. In Section \ref{section_weakly_compact_ideal_after_forcing} below, we generalize the standard fact that after $\kappa$-c.c. forcing the nonstationary ideal on $\kappa$ equals the ideal generated by the ground model nonstationary ideal on $\kappa$ by proving the following theorem which states that if $\kappa$ is a weakly compact cardinal and $\P_\kappa$ is a `typical' Easton-support iteration of length $\kappa$ then after forcing with $\P_\kappa$, the weakly compact ideal of the extension equals the ideal generated by the ground model weakly compact ideal.

\begin{theorem}\label{theorem_weakly_compact_ideal_after_forcing}
Suppose $\kappa$ is weakly compact cardinal, assume that $\P=\<(\P_\alpha,\dot{\Q}_\alpha) \st\alpha<\kappa\>\subseteq V_\kappa$ is an Easton-support forcing iteration such that for each $\alpha<\kappa$, $\forces_{\P_\alpha}$ ``$\dot{\Q}_\alpha$ is $\alpha$-strategically closed''. Let $\dot{\Pi}^1_1(\kappa)$ be a $\P_\kappa$-name for the weakly compact ideal of the extension $V^{\P_\kappa}$ and let $\check{\Pi}^1_1(\kappa)$ be a $\P_\kappa$-check name for the weakly compact ideal of the ground model. Then $\forces_{\P_\kappa}$ $\dot{\Pi}^1_1(\kappa)=\overline{\check{\Pi}^1_1(\kappa)}$. 
\end{theorem}

\section{Preliminaries}

\subsection{Weak compactness}
For our purposes it will be advantageous to work with the characterization of weak compactness stated in terms of elementary embeddings. Let us review some standard facts about such characterizations.

\begin{definition}
We say that $M$ is a \emph{$\kappa$-model} if and only if
\begin{enumerate}
\item $\kappa\in M$,
\item $|M|=\kappa$,
\item $M$ is transitive,
\item $M^{<\kappa}\cap V\subseteq M$ and 
\item $M\models\ZFC^-$.\footnote{Here $\ZFC^-$ denotes the axioms of $\ZFC$ without the powerset axiom and with the collection axiom instead of the replacement axiom \cite{MR3549557}.
}
\end{enumerate}
\end{definition}

It is easy to see that if $\kappa$ is a cardinal then the collection $H_{\kappa^+}$ of all sets whose transitive closure has size at most $\kappa$ is a model of $\ZFC^-$. Furthermore, if $X\in H_{\kappa^+}$ then, using an iterative Skolem closure argument, one can build an elementary substructure $M\prec H_{\kappa^+}$ such that $\kappa,X\in M$, $|M|=\kappa$, $M^{<\kappa}\cap V\subseteq M$ and $\kappa\subseteq M$. Since $\kappa\subseteq M$ it follows by elementarity that $M$ is transitive, and thus $M$ is a $\kappa$-model. This establishes the following.

\begin{lemma}\label{lemma_build_kappa_model}
Suppose $\kappa$ is a cardinal and $X\in H_{\kappa^+}$. There is a $\kappa$-model $M$ such that $X\in M\prec H_{\kappa^+}$.
\end{lemma}

\begin{remark}\label{remark_build_kappa_model_for_a_sequence}
Of course, if $\gamma<\kappa^+$ we can code a sequence $\vec{X}=\<X_\alpha\st\alpha<\gamma\>$ of elements of $H_{\kappa^+}$ into a single subset of $\kappa$, and thus build a $\kappa$-model $M$ with $X_\alpha\in M$ for all $\alpha<\gamma$.
\end{remark}

We take the following to be the definition of weakly compact cardinal and weakly compact set, as these are the characterizations which seem to have the most utility when working with Easton-support forcing iterations.

\begin{definition}\label{definition_weak_compactness}
Suppose $\kappa^{<\kappa}=\kappa$.
\begin{enumerate}
\item We say that $\kappa$ is a \emph{weakly compact cardinal} if for every $A\subseteq\kappa$ there is a $\kappa$-model $M$ with $A\in M$ and there is an elementary embedding $j:M\to N$ with critical point $\kappa$ where $N$ is a $\kappa$-model.
\item We say that $S\subseteq\kappa$ is a \emph{weakly compact subset of $\kappa$} if and only if for every $A\subseteq\kappa$ there is a $\kappa$-model $M$ with $A,S\in M$ and there is an elementary embedding $j:M\to N$ with critical point $\kappa$ such that $\kappa\in j(S)$ where $N$ is a $\kappa$-model.
\end{enumerate}
\end{definition}
Notice that $\kappa$ is a weakly compact cardinal if and only if it is a weakly compact subset of itself. The following lemma, essentially due to Baumgartner \cite{MR0540770}, is well known. For details see one may consult \cite[Theorem 4.13]{MR2026390}, \cite[Theorem 6.4]{Kanamori:Book} and \cite[Theorem 16.1]{Cummings:Handbook}.

\begin{lemma}\label{lemma_characterizations}
Suppose $\kappa^{<\kappa}=\kappa$. The following are equivalent. 
\begin{enumerate}
\item $S\subseteq\kappa$ is a weakly compact subset of $\kappa$.
\item $S$ is $\Pi^1_1$-indescribable; in other words, for every $A\subseteq\kappa$ and for every $\Pi^1_1$-sentence $\varphi$ if $(V_\kappa,\in,A)\models \varphi$ then there is an $\alpha\in S$ such that $(V_\alpha,\in,A\cap\alpha)\models\varphi$.
\end{enumerate}
\end{lemma}

\subsection{The weakly compact ideal}

Levy \cite{MR0281606} showed that if $\kappa$ is a weakly compact cardinal then the set
\[\Pi^1_1(\kappa)=\{X\subseteq\kappa\st \text{$X$ is not a weakly compact subset of $\kappa$}\}\]
is a normal proper ideal on $\kappa$; we call $\Pi^1_1(\kappa)$ \emph{the weakly compact ideal on $\kappa$}.
The corresponding collection of positive sets and the dual filter will be written as
\begin{align*}
\Pi^1_1(\kappa)^+&=\{S\subseteq\kappa\st \text{$S$ is a weakly compact subset of $\kappa$}\} \text{ and }\\
\Pi^1_1(\kappa)^*&=\{X\subseteq\kappa\st\text{$\kappa\setminus X$ is not a weakly compact subset of $\kappa$}\},
\end{align*}
respectively.

As shown by Sun \cite{MR1245524}, the indescribability characterization given in Lemma \ref{lemma_characterizations} (2) allows us to give an additional characterization of weakly compact subsets of a weakly compact cardinal which resembles the definition of stationarity. 

\begin{definition}[Sun, \cite{MR1245524}]\label{definition_1_closed_1_club}
Suppose $\kappa$ is a regular cardinal and $C\subseteq\kappa$. 
\begin{enumerate}
\item $C$ is \emph{$1$-closed} if and only if for every inaccessible $\gamma<\kappa$ if $C\cap\gamma$ is stationary in $\gamma$ then $\gamma\in C$.
\item We say that $C\subseteq\kappa$ is \emph{$1$-club} if and only if
\begin{enumerate}
\item $C$ is a stationary subset of $\kappa$ and
\item $C$ is $1$-closed.
\end{enumerate}
\end{enumerate}
\end{definition}

\begin{lemma}[Sun, \cite{MR1245524}]\label{lemma_1_club_characterization}
Suppose $\kappa$ is a weakly compact cardinal. Then $S\subseteq\kappa$ is a weakly compact subset of $\kappa$ if and only if $S\cap C\neq\emptyset$ for every $1$-club $C\subseteq\kappa$.
\end{lemma}
Similarly, for $n<\omega$, one may define the $n$-club subsets of $\kappa$ and prove that a set $S\subseteq\kappa$ is $\Pi^1_n$-indescribable if and only if $S\cap C\neq\emptyset$ for every $n$-club $C\subseteq\kappa$.

The forward direction of Lemma \ref{lemma_1_club_characterization} follows by a standard reflection argument. To prove the reverse direction one uses the characterization of weakly compact sets given in Lemma \ref{lemma_characterizations} (2) above and the fact that if $\varphi$ is a $\Pi^1_1$-sentence and $\kappa$ is a weakly compact cardinal then $\{\alpha<\kappa\st (V_\alpha,\in,A\cap V_\alpha)\models\varphi\}$ is $1$-club.

\subsection{Preserving weak compactness through forcing}\label{section_preserving_weak_compactness}

In the proof of our main theorem we will use the following standard lemmas to argue that various instances of weak compactness are preserved in forcing extensions. For further discussion of these methods see \cite{Cummings:Handbook}.

\begin{lemma}\label{lemma_lifting_criterion}
Suppose $j:M\to N$ is an elementary embedding with critical point $\kappa$ where $M$ and $N$ are $\kappa$-models and $\P\in M$ is some forcing notion. Suppose $G\subseteq\P$ is a filter generic over $M$ and $H\subseteq j(\P)$ is a filter generic over $N$. Then $j$ extends to $j:M[G]\to N[H]$ if and only if $j[G]\subseteq H$.
\end{lemma}

\begin{lemma}\label{lemma_generic_closure}
Suppose that $M$ is a $\kappa$-model in $V$, so in particular $M^{<\kappa}\cap V\subseteq M$, and suppose $\P\in M$ is $\kappa$-c.c. If $G\subseteq \P$ is generic over $V$, then $M[G]^{<\kappa}\cap V[G]\subseteq M[G]$.
\end{lemma}

\begin{lemma}\label{lemma_ground_closure}
Suppose that $M$ is a $\kappa$-model in $V$, so that in particular $M^{<\kappa}\cap V\subseteq M$, $\P\in M$ is some forcing notion and there is a filter $G\in V$ which is generic for $\P$ over $M$. Then $M[G]^{<\kappa}\cap V\subseteq M[G]$.
\end{lemma}

Recall that for an ordinal $\alpha$ and poset $\P$ we may define a game $\mathcal{G}_\alpha(\P)$ as follows. Players I and II take turns choosing conditions which form a descending sequence in $\P$ such that Player II must play first at limit stages and Player I makes the first move of the game. Player I wins the game if at some stage $\beta<\alpha$ there is no condition below all previously played conditions and thus no way for play to continue. Otherwise Player II wins and the plays of the game form a decreasing sequence in $\P$ of length $\alpha$. We say that $\P$ is \emph{${<}\kappa$-strategically closed} if and only if for every $\alpha<\kappa$ Player II has a winning strategy for $\mathcal{G}_\alpha(\P)$. The poset $\P$ is \emph{$\kappa$-strategically closed} if and only if Player II has a winning strategy in $\mathcal{G}_\kappa(\P)$.

\begin{lemma}\label{lemma_diagonalization_criterion}
Suppose $M$ is a $\kappa$-model in $V$, so in particular $M^{<\kappa}\cap V\subseteq M$. Furthermore, suppose $\P\in M$ is a forcing notion with $V\models$ ``$\P$ is $\kappa$-strategically closed'' and $p\in \P$. Then there is a filter $G\in V$ generic for $\P$ over $M$ with $p\in G$.
\end{lemma}

\subsection{Shooting $1$-clubs}

Hellsten proved that if $E\subseteq\kappa$ and $\kappa\setminus E$ are both weakly compact then there is a natural forcing $\H(E)$ which kills the weak compactness of $\kappa\setminus E$ by shooting a $1$-club through $E$ and preserves the weak compactness of $\kappa$ as well as every weakly compact subset of $E$. Before reviewing some of the details of Hellsten's forcing a few remarks are in order. Let us highlight a few ways in which Hellsten's forcing differs from typical club shooting forcing. Since Hellsten's forcing preserves the weak compactness of $\kappa$ and adds a cofinal subset of $\kappa$ one should expect that $\H(E)$ is an Easton-support iteration of length $\kappa+1$. Additionally, Hellsten's forcing is distinguished from club shooting forcings by the fact that no extra assumption must be made about the shape of the weakly compact set in order for $\H(E)$ to preserve cardinals.

\begin{theorem}[Hellsten, \cite{MR2026390}]\label{theorem_1_club_shooting}
Suppose $E$ is a weakly compact subset of $\kappa$. There is a forcing extension in which $E$ contains a $1$-club, all weakly compact subsets of $E$ remain weakly compact and thus $\kappa$ remains a weakly compact cardinal.
\end{theorem}
Let us review some of the details of Hellsten's forcing and the salient points of the proof of Theorem \ref{theorem_1_club_shooting} which will be relevant to our proof of Theorem \ref{theorem_main_theorem}.
Let $X\subseteq\kappa$ be an unbounded subset of an inaccessible cardinal. We define a poset\[T^1(X)=\{c\subseteq X\st\text{$c$ is bounded and $1$-closed}\}\]
ordered by end extension: $c_1\leq c_2$ iff $c_2=c_1\restrict\sup\{\alpha+1\st\alpha\in c_2\}$.

One can easily see that the poset $T^1(X)$ is ${<}\kappa$-strategically closed\footnote{See Section \ref{section_preserving_weak_compactness} for our conventions on strategic closure terminology.} and thus does not add bounded subsets to $\kappa$ by observing that for any $\alpha<\kappa$, the set $D_\alpha=\{c\in T^1(X)\st \sup c >\alpha^+\}$ is an $\alpha^+$-closed dense subset of $T^1(X)$. Moreover, Hellsten proved \cite[Lemma 3]{MR2653962} that $T^1(X)$ is $\kappa$-strategically closed. 
\begin{lemma}[Hellsten, \cite{MR2653962}]\label{lemma_1club_shooting_str_closed}
If $\kappa$ is an inaccessible cardinal and $X\subseteq\kappa$ is unbounded then $T^1(X)$ is $\kappa$-strategically closed. 
\end{lemma}

The forcing $\H(E)$ which Hellsten used to prove Theorem \ref{theorem_1_club_shooting} is an Easton-support iteration $\<\P_\alpha,\dot{\Q}_\beta\st \alpha\leq\kappa+1,\beta\leq\kappa\>$ such that
\begin{enumerate}
\item if $\beta\leq\kappa$ is Mahlo then $\dot{\Q}_\beta$ is a $\P_\beta$-name for $T^1(E\cap\beta)^{V^{\P_\beta}}$ and
\item otherwise $\dot{\Q}_\beta$ is a $\P_\beta$-name for trivial forcing.
\end{enumerate}

Hellsten's proof of Theorem \ref{theorem_1_club_shooting} in \cite{MR2026390} uses both the elementary embedding characterization and the $\Pi^1_1$-indescribability characterization of weak compactness. Here we give a proof of Theorem \ref{theorem_1_club_shooting} which only uses the elementary embedding characterization of weak compactness as it will be relevant later in our proof of Theorem \ref{theorem_main_theorem}.

\begin{proof}[Proof of Theorem \ref{theorem_1_club_shooting}]
Let $G_{\kappa+1}\cong G_\kappa*H_\kappa$ be generic for $\H(E)\cong\P_\kappa*\dot{T}^1(E)$ over $V$. Working in $V[G_{\kappa+1}]$, let $C(\kappa)=\bigcup H_\kappa$. In $V[G_\kappa]$, conditions in $T^1(E)$ are bounded $1$-closed subsets of $E$ and thus it follows trivially that in $V[G_{\kappa+1}]$ the set $C(\kappa)$ is a $1$-closed subset of $E$.

We will show simultaneously that in $V[G_{\kappa+1}]$, the set $C(\kappa)$ is stationary in $\kappa$ and that every weakly compact subset of $E$ remains weakly compact. Suppose $\dot{C}\in H_{\kappa^+}$ is a $\P_{\kappa+1}$-name for a club subset of $\kappa$ and $S\subseteq E$ is a weakly compact subset of $\kappa$ in $V$. Working in $V[G_{\kappa+1}]$, suppose $A\subseteq\kappa$ and let $\dot{A}\in H_{\kappa^+}$ be a $\P_{\kappa+1}$-name such that $\dot{A}_{G_{\kappa+1}}=A$. Let $M$ be a $\kappa$-model with $\dot{A},\dot{C},\dot{C}(\kappa),S,E\in M$ and let $j:M\to N$ be an elementary embedding with critical point $\kappa$ such that $\kappa\in j(S)$ where $N$ is also a $\kappa$-model. By elementarity, in $N$, $j(\<\P_\alpha,\dot{\Q}_\beta\st \alpha\leq\kappa+1,\beta\leq\kappa\>)$ is an Easton support iteration of length $j(\kappa+1)$ which we denote by $\<\P_\alpha',\dot{\Q}_\beta'\st\alpha\leq j(\kappa+1),\beta\leq j(\kappa)\>$. Since $N^{<\kappa}\cap V\subseteq N$ it follows that $\P_{\kappa+1}'=\P_{\kappa+1}$. By Lemma \ref{lemma_generic_closure}, $N[G_\kappa]^{<\kappa}\cap V[G_\kappa]\subseteq N[G_\kappa]$ and since $T^1(E)$ is $\kappa$-strategically closed in $V[G_\kappa]$, it follows that $N[G_\kappa*H_\kappa]^{<\kappa}\cap V[G_\kappa*H_\kappa]\subseteq N[G_\kappa*H_\kappa]$. Thus by Lemma \ref{lemma_diagonalization_criterion}, we may build a generic filter $K$ for $\P'_{\kappa,j(\kappa)}$ over $N[G_\kappa*H_\kappa]$. Let $\hat{G}_{j(\kappa)}=G_\kappa*H_\kappa*K$. Since conditions in $\P_\kappa$ have support bounded below the critical point of $j$ we have $j[G_\kappa]\subseteq\hat{G}_{j(\kappa)}$ and thus by Lemma \ref{lemma_lifting_criterion}, $j$ extends to $j:M[G_\kappa]\to N[\hat{G}_{j(\kappa)}]$.

Since $\kappa\in j(S)\subseteq j(E)$ it follows that $c=C(\kappa)\cup\{\kappa\}$ is a condition in $j(T^1(E))$. Since $j(T^1(E))$ is $\kappa$-strategically closed in $N[\hat{G}_{j(\kappa)}]$ and since $N[\hat{G}_{j(\kappa)}]$ is a $\kappa$-model in $V[G_{\kappa+1}]$, we may apply Lemma \ref{lemma_diagonalization_criterion} to find $\hat{H}_{j(\kappa)}$ a generic filter for $j(T^1(E))$ with $c\in\hat{H}_{j(\kappa)}$ and $\hat{H}_{j(\kappa)}\in V[G_{\kappa+1}]$. Since $c$ is below every condition in $j[H_\kappa]$ we have $j[H_\kappa]\subseteq\hat{H}_{j(\kappa)}$ and thus $j$ extends to $j:M[G_\kappa*H_\kappa]\to N[\hat{G}_{j(\kappa)}*\hat{H}_{j(\kappa)}]$. Since $\kappa\in j(S)$ it follows that $S$ is a weakly compact subset of $\kappa$ in $V[G_{\kappa+1}]$. Furthermore, since $c\in \hat{H}_{j(\kappa)}$, we have $\kappa\in j(C(\kappa)\cap \dot{C}_{G_{\kappa+1}})$ and thus $C(\kappa)\cap \dot{C}_{G_{\kappa+1}}\neq\emptyset$. Hence $C(\kappa)$ is a stationary subset of $\kappa$ in $V[G_{\kappa+1}]$.
\end{proof}

\subsection{Lottery sums}

The forcing we will use to prove Theorem \ref{theorem_main_theorem} will be an Easton support iteration of lottery sums. For the reader's convenience let us review the definition and some basic properties of the lottery sum construction. 

\begin{definition}
For each $\alpha<\eta$ suppose $\R_\alpha=(\R_\alpha,1_\alpha,\leq_\alpha)$ is a poset with greatest element $1_\alpha$. As in \cite{Hamkins:TheLotteryPreparation}, we define the lottery sum $\bigoplus_{\alpha<\eta}\R_\alpha=\bigoplus\{\R_\alpha\st \alpha<\eta\}$ to be the poset whose underlying set is
\[\bigoplus_{\alpha<\eta}\R_\alpha:=\{1\}\cup\left(\bigsqcup_{\alpha<\eta}\R_\alpha\right)=\{1\}\cup\left(\bigcup_{\alpha<\eta}(\{\alpha\}\times\R_\alpha)\right)\]
and which is ordered by letting $1$ be the greatest (weakest) element and defining $(\alpha,r)\leq (\beta,s)$ if and only if $\alpha=\beta$ and $r\leq_\alpha s$. Using terminology of Hamkins, we say that a condition $p$ in $\bigoplus_{\alpha<\eta}\R_\alpha$ \emph{opts for $\R_\alpha$} if and only if $p$ is of the form $p=(\alpha,r)$ for some $\alpha<\eta$ and some $r\in\R_\alpha$.
\end{definition}
If $G$ is generic for $\bigoplus_{\alpha<\eta}\R_\alpha$ then there is a unique $\alpha<\eta$ such that $(\alpha,1_\alpha)\in G$, and for that $\alpha$ the set $\{r\in\R_\alpha\st (\alpha,r)\in G\}$ is generic for $\R_\alpha$ over $V$. In this way, by forcing with $\bigoplus_{\alpha<\eta}\R_\alpha$ we are generically selecting one of the posets $\R_\alpha$ to force with.

\subsection{Orders of weak compactness}\label{section_orders}

Suppose $\kappa$ is a weakly compact cardinal. Elements of the boolean algebra $P(\kappa)/\Pi^1_1(\kappa)$ are written as $[X]_1$ where $X\in P(\kappa)$. We can define an operation $\Tr_{\wc}:P(\kappa)\to P(\kappa)$ analogous to the Mahlo operation by letting 
\[\Tr_{\wc}(X)=\{\gamma<\kappa\st\text{$X\cap\gamma$ is a weakly compact subset of $\gamma$}\}.\]
It is well known that if $I$ is a normal ideal on $\kappa$ then the diagonal intersection of a family of $\leq\kappa$-many subsets of $\kappa$ is independent of the indexing used modulo $I$ \cite{MR0505505}. This, together with the fact that $\Tr_{\wc}$ is a well-defined operation on $P(\kappa)/\Pi^1_1(\kappa)$ \cite{MR2252250}, allows one to calculate $\kappa^+$ iterates of $\Tr_{\wc}$ as an operation on $P(\kappa)/\Pi^1_1(\kappa)$. In summary we define
\begin{align*}
\Tr^0_{\wc}([X]_1)&=[X]_1,\\
\Tr^{\xi+1}_{\wc}([X]_1)&=\Tr_{\wc}(\Tr^\xi_{\wc}([X]_1))\text{ and }\\
\Tr^\xi_{\wc}([X]_1)&=\bigtriangleup \{\Tr^\eta_{\wc}([X]_1)\st \eta<\xi\}\text{ if $\xi<\kappa^+$ is a limit.}
\end{align*}
Notice that at limits $\xi<\kappa$ we have $\Tr^\xi_{\wc}([X]_1)=\bigcap_{\eta<\xi}\Tr^\eta_{\wc}([X]_1)$.

\begin{definition}\label{definition_alpha_weakly_compact}
Suppose $\kappa$ is a weakly compact cardinal. We say that $E\subseteq\kappa$ is $\xi$-weakly compact where $\xi<\kappa^+$ if and only if for every $\eta<\xi$ we have $\Tr^\eta_{\wc}([E]_1)\neq[\emptyset]_1=\Pi^1_1(\kappa)$. Of course, $\kappa$ is \emph{$\xi$-weakly compact} if it is an $\xi$-weakly compact subset of itself.
\end{definition} 
For more details on orders of weak compactness, and more generally orders of indescribability, see \cite{Cody:AddingANonReflectingWeaklyCompactSet} and \cite{MR2252250}.

\begin{lemma}
The weakly compact reflection principle $\Refl_{\wc}(\kappa)$ implies that $\kappa$ is $\omega$-weakly compact.
\end{lemma}

\begin{proof}
$\Refl_{\wc}(\kappa)$ implies that the collection $\Pi^1_1(\kappa)^+$ of weakly compact subsets of $\kappa$ is closed under the weakly compact trace operation $\Tr_{\wc}$. Thus $\Tr_{\wc}^n([\kappa]_1)\neq[\emptyset]_1$ for all $n<\omega$.
\end{proof}

In what follows we will use the following characterization of the $\omega$-weak compactness of a cardinal. If $\beta$ is a cardinal we define
\[S^\beta_0=\{\alpha<\beta\st \text{$\alpha$ is Mahlo}\}.\]
For an ordinal $\xi$ with $0<\xi<\beta$ we define
\[S^\beta_\xi=\{\alpha<\beta\st\text{$\alpha$ is $\xi$-weakly compact}\}.\]

\begin{lemma}\label{lemma_characterization_of_omega_weak_compactness}
$\kappa$ is $\omega$-weakly compact if and only if $S^\kappa_n\in\Pi^1_1(\kappa)^+$ for every $n<\omega$.
\end{lemma}

\begin{remark}
Using Lemma \ref{lemma_1_club_characterization} and Lemma \ref{lemma_characterization_of_omega_weak_compactness}, it follows that $\kappa$ is $\omega$-weakly compact if and only if $S^\kappa_n\cap C\neq\emptyset$ for every $1$-club $C\subseteq\kappa$.
\end{remark}

\section{The weakly compact ideal after forcing}\label{section_weakly_compact_ideal_after_forcing}

Recall that if $\P$ is a $\kappa$-c.c. forcing notion then it follows that after forcing with $\P$ the nonstationary ideal on $\kappa$ equals the ideal generated by the ground model nonstationary ideal
\[\forces_{\P} \text{$\dot{\NS}_\kappa=\overline{\check{\NS}_\kappa}$}.\]
To see that $\forces_{\P} \text{$\dot{\NS}_\kappa\supseteq\overline{\check{\NS}_\kappa}$}$ just note that if $\dot{X}$ is forced to have a ground-model cover $Y\in V\cap \NS_\kappa^V$, then any club $C\subseteq\kappa\setminus Y$ witnesses that $\dot{X}$ is forced to be nonstationary in $V^\P$. Conversely, to see that $\forces_{\P} \text{$\dot{\NS}_\kappa\subseteq\overline{\check{\NS}_\kappa}$}$, notice that every club in the extension contains a ground-model club.

We would like to generalize the previously mentioned result about $\kappa$-c.c. forcing and the nonstationary ideal on $\kappa$ to the weakly compact ideal. For what forcings $\P$ do we have $\forces_{\P}$ $\dot{\Pi}^1_1(\kappa)=\overline{\check{\Pi}^1_1(\kappa)}$? 
Kunen showed that in certain settings there is a $\kappa$-c.c. $\P$ that can turn a ground model non weakly compact cardinal into a weakly compact cardinal. However, if we assume $\kappa$ is a weakly compact cardinal in the ground model then non weakly compact subsets $\kappa$ cannot become weakly compact after $\kappa$-c.c. forcing.\footnote{Thanks to Sean Cox and Monroe Eskew for pointing this out.}

\begin{lemma}\label{lemma_weakly_compact_ideal_after_kappa_cc}
Suppose $\kappa$ is a weakly compact cardinal, $\P$ is $\kappa$-c.c. and $G$ is generic for $\P$ over $V$. Then $\Pi^1_1(\kappa)^V\subseteq \Pi^1_1(\kappa)^{V[G]}$.
\end{lemma}

\begin{proof}
Suppose $X\in\Pi^1_1(\kappa)^V$, then there is a $1$-club $C$ in $V$ such that $C\subseteq\kappa\setminus X$. It will suffice to show that $C$ remains $1$-club in $V[G]$. Since $\P$ is $\kappa$-c.c. it follows that $C$ is stationary in $V[G]$. Now suppose that in $V[G]$ the set $C\cap \mu$ is stationary in $\mu$ for some inaccessible cardinal $\mu<\kappa$. Since $V\subseteq V[G]$ we see that $C\cap \mu$ is stationary in $V$ and since $C$ is $1$-club in $V$, $\mu\in C$. Thus in $V[G]$ there is a $1$-club disjoint from $X$ and hence $X\in\Pi^1_1(\kappa)^{V[G]}$.
\end{proof}

Let us now show that by the work of Hamkins \cite{MR2063629}, if we assume that $\kappa$ is an inaccessible cardinal in the ground model and $\P$ is a forcing which \emph{admits a gap below $\kappa$}, in the sense that there is a $\delta<\kappa$ such that $\P\cong\Q*\dot{\R}$ where $\Q$ is $\delta$-c.c. and $\forced_\Q$ ``$\dot{\R}$ is $\delta$-strategically closed'', then ground model non weakly compact subsets of $\kappa$ cannot become weakly compact after forcing with $\P$. Let us recall two definitions from \cite{MR2063629}. A pair of transitive classes $M\subseteq N$ satisfies the \emph{$\delta$-approximation property} if whenever $A\subseteq M$ is a set in $N$ and $A\cap a\in M$ for any $a\in M$ of size less than $\delta$ in $M$, then $A\in M$. The pair $M\subseteq N$ satisfies the \emph{$\delta$-cover property} if for every set $A$ in $N$ with $A\subseteq M$ and $|A|^N<\delta$, there is a set $B\in M$ with $A\subseteq B$ and $|B|^M<\delta$. One can easily see that many Easton-support iterations $\P$ of length greater than a Mahlo cardinal $\delta$ satisfy the $\delta$-approximation and cover properties by factoring the iteration as $\P\cong\Q*\dot{\R}$ where $\Q$ is $\delta$-c.c. and $\dot{\Q}$ is ${<}\delta$-strategically closed.

\begin{lemma}[Hamkins, \cite{MR2063629}]\label{lemma_approximation_and_cover}
Suppose that $\kappa$ is an inaccessible cardinal, $S\in P(\kappa)^V$ and $V\subseteq\overline{V}$ satisfies the $\delta$-approximation and cover properties for some $\delta<\kappa$. If $S$ is a weakly compact subset of $\kappa$ in $\overline{V}$ then $S$ is a weakly compact subset of $\kappa$ in $V$.
\end{lemma}

\begin{proof}
Suppose $S\in P(\kappa)^V$ is a weakly compact subset of $\kappa$ in $\overline{V}$. Fix $A\in P(\kappa)^V$. By \cite[Lemma 16]{MR2063629}, there is a transitive model $\overline{M}\in\overline{V}$ of some large fixed finite fragment $\ZFC^*$ of $\ZFC$ with $|\overline{M}|^{\overline{V}}=\kappa$ such that $\kappa,A,S\in\overline{M}$, the model $\overline{M}$ is closed under ${<}\kappa$-sequences from $\overline{V}$ and $M=\overline{M}\cap V\in V$ is a transitive model of the finite fragment $\ZFC^*$ with $|M|^V=\kappa$. Since $S$ is weakly compact in $\overline{V}$, it follows that there is an elementary embedding $j:\overline{M}\to\overline{N}$ with critical point $\kappa$ where $\overline{N}^{<\kappa}\cap\overline{V}\subseteq\overline{N}$ and $\kappa\in j(S)$. Since this embedding satisfies the hypotheses of the main theorem from \cite{MR2063629}, it follows that $j\restrict M:M\to N$ is an elementary embedding in $V$ with critical point $\kappa$. Since $A,S\in M$ and $\kappa\in (j\restrict M)(S)$ we see that $S$ is a weakly compact subset of $\kappa$ in $V$.
\end{proof}

The following theorem, which was stated in the introduction, is a generalization to the weakly compact ideal of the fact that after $\kappa$-c.c. forcing the nonstationary ideal of the extension equals the ideal generated by the ground model weakly compact ideal.

\begin{theorem3}
Suppose $\kappa$ is weakly compact cardinal, assume that $\P=\<(\P_\alpha,\dot{\Q}_\alpha)\st\alpha<\kappa\>\subseteq V_\kappa$ is an Easton-support forcing iteration such that for each $\alpha<\kappa$, $\forces_{\P_\alpha}$ ``$\dot{\Q}_\alpha$ is $\alpha$-strategically closed''. Let $\dot{\Pi}^1_1(\kappa)$ be a $\P_\kappa$-name for the weakly compact ideal of the extension $V^{\P_\kappa}$ and let $\check{\Pi}^1_1(\kappa)$ be a $\P_\kappa$-check name for the weakly compact ideal of the ground model. Then $\forces_{\P_\kappa}$ $\dot{\Pi}^1_1(\kappa)=\overline{\check{\Pi}^1_1(\kappa)}$. 
\end{theorem3}

\begin{proof}
Since $\P_\kappa$ is $\kappa$-c.c., $\forces_{\P_\kappa}$ $\dot{\Pi}^1_1(\kappa)\supseteq \overline{\check{\Pi}^1_1(\kappa)}$ follows from Lemma \ref{lemma_weakly_compact_ideal_after_kappa_cc}.

To show that $\forces_{\P_\kappa}$ $\dot{\Pi}^1_1(\kappa)\subseteq\overline{\check{\Pi}^1_1(\kappa)}$, we will show that if $p\forces_{\P_\kappa}$ $\dot{X}\in \dot{\Pi}^1_1(\kappa)$, then there is $B\in\Pi^1_1(\kappa)^V$ with $p\forces_{\P_\kappa}$ $\dot{X}\subseteq B$. Suppose $p\forces_{\P_\kappa}$ $\dot{X}\in \dot{\Pi}^1_1(\kappa)$. Since $\P_\kappa\subseteq V_\kappa$ is $\kappa$-c.c. we may assume that $\dot{X}\in H_{\kappa^+}$. By the fullness principle, take a $\P_\kappa$-name $\dot{A}\in H_{\kappa^+}$ for a subset of $\kappa$ such that
\begin{align*}
p\forces_{\P_\kappa} & \text{``for every $\kappa$-model $M$ with $\kappa,\dot{A},\dot{X}\in M$ and for every}\tag{$*$}\\
 & \text{elementary embedding $j:M\to N$ where $N$ is a $\kappa$-model}\\
 & \text{we have $\kappa\notin j(\dot{X})$''}
\end{align*}
Let $B=\{\alpha<\kappa\st\exists q\in\P_\kappa\ (q\leq p)\land (q\forces_{\P_\kappa} \alpha\in \dot{X})\}$ and notice that $B\in V$ and $p\forces_{\P_\kappa}$ $\dot{X}\subseteq B$. Thus, to complete the proof it will suffice to show that $B\in\Pi^1_1(\kappa)^V$.

Suppose $B\notin\Pi^1_1(\kappa)^V$. Using the weak compactness of $B$ in $V$, let $M$ be a $\kappa$-model with $\kappa,B,\dot{A},\P_\kappa,\dot{X},p,\ldots\in M$ and let $j:M\to N$ be an elementary embedding with critical point $\kappa$ such that $\kappa\in j(B)$ where $N$ is a $\kappa$-model. Since $\kappa\in j(B)$, it follows by elementarity that there is a condition $r\in j(\P_\kappa)$ with $r\leq j(p)=p$ such that $r\forces_{j(\P_\kappa)}\kappa\in j(\dot{X})$. Let $G\subseteq\P_\kappa$ be generic over $V$ with $r\restrict\kappa\in G$. By Lemma \ref{lemma_generic_closure}, since $\P_\kappa$ is $\kappa$-c.c. the model $N[G]$ is closed under ${<}\kappa$-sequences in $V[G]$. Furthermore, the poset $j(\P_\kappa)/G$ is $\kappa$-strategically closed in $N[G]$. Thus, by Lemma \ref{lemma_diagonalization_criterion}, working in $V[G]$ we can build a filter $H\subseteq j(\P_\kappa)/G$ which is generic over $N[G]$ with $r/G\in H$. Let $\hat{G}$ denote the filter for $j(\P_\kappa)$ obtained from $G*H$ and notice that $r\in \hat{G}$. Since conditions in $\P_\kappa$ have support bounded below the critical point of $j$, it follows that $j[G]\subseteq \hat{G}$. Thus the embedding extends to $j:M[G]\to N[\hat{G}]$. Since $r\in \hat{G}$ and $r\forces_{j(\P_\kappa)}$ $\kappa\in j(\dot{X})$, we have $\kappa\in j(\dot{X}_G)$. Notice that $p\in G$; this contradicts ($*$) since $M[G]$ and $N[\hat{G}]$ are $\kappa$-models and since $A=\dot{A}_G\in M[G]$.
\end{proof}

\section{Forcing the weakly compact reflection principle to hold at the least $\omega$-weakly compact cardinal}

Next we begin our proof of Theorem \ref{theorem_main_theorem}. Assuming the weakly compact reflection principle $\Refl_{\wc}(\kappa)$ holds we will define an Easton-support forcing iteration of length $\kappa$ which kills every $\omega$-weakly compact cardinal below $\kappa$, preserves the weakly compact reflection principle $\Refl_{\wc}(\kappa)$ and thus preserves the $\omega$-weak compactness of $\kappa$.

\begin{proof}[Proof of Theorem \ref{theorem_main_theorem}]

First we define the forcing iteration. If $\beta$ is a cardinal, recall that we defined $S^\beta_0=\{\alpha<\beta\st\text{$\alpha$ is a Mahlo cardinal}\}$ and
for an ordinal $0<\xi<\beta$ we have defined
\[S^\beta_\xi=\{\alpha<\beta\st \text{$\alpha$ is $\xi$-weakly compact}\}.\]
For each $n<\omega$ we let $T^1(\beta\setminus S^\beta_n)$ be the poset whose conditions are bounded $1$-closed subsets of $\beta\setminus S^\beta_n$ ordered by end extension. For each $m<\omega$ we let $\prod_{n\in[m,\omega)}T^1(\beta\setminus S^\beta_n )$ be the product forcing whose conditions are functions $f$ with domain $[m,\omega)$ such that $f(n)\in T^1(\beta\setminus S^\beta_n)$ for each $n\in[m,\omega)$, and whose the greatest element is $1_m:[m,\omega)\to\{\emptyset\}$. Notice that $1_m$ is the greatest (weakest) element of $\prod_{n\in[m,\omega)}T^1(\beta\setminus S^\beta_n )$ \emph{for all} $\beta$. Next we define a forcing iteration of length $\kappa$ which will shoot $1$-clubs through subsets of each weakly compact $\gamma<\kappa$.

Let $\P_\kappa=\< \P_\alpha,\dot{\Q}_\beta \st \alpha\leq\kappa, \beta<\kappa\>$ be the Easton support iteration of length $\kappa$ defined as follows.
\begin{enumerate}
\item If $\beta<\kappa$ is Mahlo in $V$ then $\dot{\Q}_\beta$ is a $\P_\beta$-name such that 
\[\forces_{\P_\beta} \dot{\Q}_\beta=\bigoplus_{m<\omega}\prod_{n\in[m,\omega)}T^1(\beta\setminus \dot{S}^\beta_n ),\]where $\dot{S}^\beta_n$ is a $\P_\beta$-name for the set of $n$-weakly compact cardinals less than $\beta$ as defined in $V^{\P_\beta}$.
\item Otherwise, $\dot{\Q}_\beta$ is a $\P_\beta$-name for trivial forcing.
\end{enumerate}

\begin{remark}\label{remark_intuitive_comment}
We first attempted to prove Theorem \ref{theorem_main_theorem} using an Easton-support iteration $\bar{\P}_\kappa=\<\bar{\P}_\alpha,\dot{\bar{\Q}}_\alpha\st \alpha\leq\kappa, \beta<\kappa\>$ such that when $\beta<\kappa$ is Mahlo $\dot{\bar{\Q}}_\beta$ is a $\bar{\P}_\beta$-name such that $\forces_{\bar{\P}_\beta}$ ``$\dot{\bar{\Q}}_\beta=\bigoplus_{m<\omega}T^1(\beta\setminus S^\beta_n)$'', and otherwise $\dot{\bar{\Q}}_\beta$ is a $\bar{\P}_\beta$-name for trivial forcing. However, we could not prove Lemma \ref{lemma_hiroshi} for this simpler iteration. It seems that the use of the lottery sum \emph{of products} is necessary in order for a certain master condition to exist.
\end{remark}

We will prove that $\P_\kappa$ witnesses Theorem \ref{theorem_main_theorem}.
For this it suffices to prove the following:

\begin{lemma}\label{lemma_no_omega_wc}
In $V^{\P_\kappa}$ there are no $\omega$-weakly compact cardinals below $\kappa$.
\end{lemma}

\begin{lemma}\label{lemma_refl_P_kappa}
$\Refl_{\wc}(\kappa)$ holds in $V^{\P_\kappa}$.
\end{lemma}

\noindent
Here note that the $\omega$-weakly compactness of $\kappa$ in $V^{\P_\kappa}$ follows from $\Refl_{\wc}(\kappa)$.

First we prove Lemma \ref{lemma_no_omega_wc}.

\begin{proof}[Proof of Lemma \ref{lemma_no_omega_wc}]
Suppose that $G_\kappa\subseteq\P_\kappa$ is a filter generic over $V$ and that $\gamma < \kappa$. Working in $V[ G_\kappa ]$, we prove that $\gamma$ is not $\omega$-weakly compact. Note that if $\gamma$ is not Mahlo in $V$, then it is not Mahlo in $V[ G_\kappa ]$. So we assume that $\gamma$ is Mahlo in $V$.

Let $G_\gamma := G_\kappa \cap \P_\gamma$ and $H_\gamma$ be a $(\dot{\Q}_\gamma)_{G_\gamma}$-generic filter over $V[ G_\gamma ]$ which is naturally obtained from $G_\kappa$. Then $H_\gamma = \{ m \} \times I_\gamma$ for some $m < \omega$, where $I_\gamma$ is a $\prod_{n\in [m,\omega)} T^1 ( \gamma \setminus S^\gamma_n )$-generic filter
over $V[ G_\gamma ]$. Let $J_\gamma := \{ f(m) \st f \in I_\gamma \}$ and $C := \bigcup J_\gamma$.

Then $C \cap (S^\gamma_m)^{V[ G_\gamma ]} = \emptyset$.
But note that $(S^\gamma_n)^{V[ G_\kappa ]} = (S^\gamma_n)^{V[ G_\gamma ]}$ for all $n < \omega$, which can be easily proved by induction on $n$ using the fact that $V[ G_\gamma ]^{< \gamma} \cap V[ G_\kappa ] \subseteq V[ G_\gamma ]$. Thus $C \cap (S^\gamma_m)^{V[ G_\kappa ]} = \emptyset$. So it suffices to prove that $C$ is $1$-club in $\gamma$ in $V[ G_\kappa ]$. It easily follows from the definition of $T^1 (\gamma\setminus S^\gamma_m)$ that $C$ is $1$-closed in $\gamma$. To see that $C$ is stationary in $\gamma$ in $V[ G_\kappa ]$, first notice that $C$ is stationary in $\gamma$ in $V[ G_\gamma * I_\gamma ]$. This can be proved by a standard density argument. Then, since $V[ G_\gamma * I_\gamma ]^\gamma \cap V[ G_\kappa ] \subseteq V[ G_\gamma * I_\gamma ]$, we have that $C$ is stationary in $\gamma$ in $V[ G_\kappa ]$.
\end{proof}

In order to prove Lemma \ref{lemma_refl_P_kappa} we will need the following.

\begin{lemma}\label{lemma_hiroshi}
Suppose that $\gamma<\kappa$, $p\in\P_\gamma$ and $\dot{S}$ is a $\P_\gamma$-name for a subset of $\gamma$. Suppose that 
\[p\forces_{\P_\gamma} \text{``$\dot{S}$ is a weakly compact subset of $\gamma$''.}\]
Then there is $r\leq p$ with $r\in\P_{\gamma+1}$ such that
\[r\forces_{\P_{\gamma+1}} \text{``$\dot{S}$ is a weakly compact subset of $\gamma$''};\]
in other words, $p$ does not force $\dot{S}$ to be a non-weakly compact subset of $\gamma$ in $V^{\P_{\gamma+1}}$.
\end{lemma}

\begin{proof}
%

For the sake of contradiction, assume that
\[p\forces_{\P_{\gamma+1}} \text{``$\dot{S}$ is not a weakly compact subset of $\gamma$''}.\]
By the fullness principle there is a $\P_{\gamma+1}$-name $\dot{A}\in H_{\gamma^+}$ for a subset of $\gamma$ such that
\begin{align*}
p\forces_{\P_{\gamma+1}} & \text{``for every $\gamma$-model $M$ with $\gamma,\dot{A},\dot{S}\in M$ and for every}\tag{$*$}\\
 & \text{elementary embedding $j:M\to N$ where $N$ is a $\gamma$-model}\\
 & \text{we have $\gamma\notin j(\dot{S})$''}
\end{align*}
We will contradict $(*)$ by finding a filter $G_{\gamma+1}\subseteq\P_{\gamma+1}$ containing $p$ such that in $V[G_{\gamma+1}]$ there are $\gamma$-models $M^*,N^*$ with $\gamma,\dot{A}_{G_{\gamma+1}},\dot{S}_{G_{\gamma+1}}\in M^*$ and an elementary embedding $j:M^*\to N^*$ with critical point $\gamma$ such that $\gamma\in j(\dot{S}_{G_{\gamma+1}})$.

Let $B=\{\alpha<\kappa\st\exists q\in\P_\gamma\ (q\leq p)\land (q\forces_{\P_\gamma} \alpha\in \dot{S})\}$. Since $p\forces_{\P_\gamma}$ ``$\dot{S}\subseteq B$'', it follows that $p\forces_{\P_\gamma}$ ``$B$ is weakly compact''. By Lemma \ref{lemma_approximation_and_cover}, $B$ is a weakly compact subset of $\gamma$ in $V$. Let $M$ be a $\gamma$-model with $p,\P_{\gamma+1},\dot{S},\dot{A},B \in M$ and let $j:M\to N$ be an elementary embedding with critical point $\gamma$ such that $\gamma\in j(B)$ where $N$ is a $\gamma$-model.

By elementarity, $j(\<\P_\alpha,\dot{\Q}_\beta\st \alpha\leq\gamma+1,\beta\leq\gamma\>)$ is an Easton support iteration of length $j(\gamma)+1$, which we denote by $\<\P'_\alpha,\dot{\Q}'_\alpha\st\alpha\leq j(\gamma)+1,\beta\leq j(\gamma)\>$. Since $N$ is a $\gamma$-model we have $\P_{\gamma+1}=\P'_{\gamma+1}$. Since $\gamma\in j(B)$, from the definition of $B$ we see that there is a condition $r'\leq j(p)=p$ in $j(\P_\gamma)=\P'_{j(\gamma)}$ such that $r'\forces_{j(\P_\gamma)}^N$ $\gamma\in j(\dot{S})$. Here note that $\gamma$ is Mahlo in $N$. So, by extending $r'$ if necessary, we may assume that $r'\restrict\gamma\forces_{\P_\gamma}^N$ $r'(\gamma)=(m,\dot{f})$ for some $m < \omega$ and some $\dot{f}\in(\prod_{n\in[m,\omega)}T^1(\gamma\setminus S^\gamma_n))^{V^{\P_\gamma}}$. Define a condition $r\in \P_\gamma*\dot{\Q}_\gamma$ by letting $r=r'\restrict\gamma\cup\{(\gamma,(m+1,\dot{f}\restrict[m+1,\omega))\}$. Then $r$ opts for the product $(\prod_{n\in[m+1,\omega)}T^1(\gamma\setminus S^\gamma_n))^{V^{\P_\gamma}}$ in the stage $\gamma$ lottery sum.

Let $G_{\gamma+1}\cong G_\gamma*H_\gamma$ be a filter for $\P_{\gamma+1} = \P_\gamma*\dot{\Q}_\gamma$ generic over $V$ with $r\in G_\gamma * H_\gamma$. Notice that $p\in G_{\gamma+1}$ since $r'\restrict\gamma\leq p$. Let $A$, $S$ and $f$ be the evaluations of $\dot{A}$, $\dot{S}$ and $\dot{f}$ by $G_{\gamma+1}$, respectively. Then $A,S \in M[G_{\gamma+1}]$. Notice also that $M[G_{\gamma+1}]$ is a $\gamma$-model in $V[G_{\gamma+1}]$ by Lemma \ref{lemma_generic_closure} and the fact that $(\dot{\Q}_\gamma)_{G_\gamma}$ is $\gamma$-strategically closed. Thus, to contradict $(*)$, it suffices to prove that in $V[G_{\gamma+1}]$, $j$ can be extended to an elementary embedding $j:M[G_{\gamma+1}] \to N^*$ with $\gamma\in j(S)$ for some $\gamma$-model $N^*$. $N^*$ will be an extension of $N$ by $\P_{j(\gamma)}'$. We will work in $V[ G_{\gamma+1} ]$ below.

We have $H_\gamma=\{m+1\}\times I_\gamma$ where $I_\gamma$ is a generic filter for $\prod_{n\in[m+1,\omega)}T^1(\gamma\setminus S^\gamma_n)$ over $V[G_\gamma]$. We define a filter $H_\gamma'\subseteq\Q_\gamma'=(\dot{\Q}_\gamma')_{G_\gamma}$ generic over $N[G_\gamma]$ as follows. By Lemma \ref{lemma_generic_closure}, $N[G_\gamma]^{<\gamma}\cap V[G_\gamma]\subseteq N[G_\gamma]$ and moreover $V[G_\gamma]$ believes that $N[G_\gamma]$ is a $\gamma$-model. Since $T^1(\gamma\setminus S^\gamma_m)$ is $\gamma$-strategically closed in $N[G_\gamma]$ it follows from Lemma \ref{lemma_diagonalization_criterion} that there is a filter $J_\gamma\in V[G_\gamma]$ which is generic for $T^1(\gamma\setminus S^\gamma_m)$ over $N[G_\gamma]$ such that $f(m)\in J_\gamma$. Thus $H_\gamma':=\{m\}\times J_\gamma\times I_\gamma$ is generic for $\Q_\gamma'$ over $N[G_\gamma]$ and $r'(\gamma)=(m,f)\in H_\gamma'$. Let $\hat{G}_{\gamma+1} := G_\gamma * H_\gamma'$.

Since $\Q_\gamma$ is $\gamma$-strategically closed, we have $N[G_\gamma]^{<\gamma}\cap V[G_{\gamma+1}]\subseteq N[G_\gamma]$, and since $H_\gamma ' \in V[G_{\gamma+1}]$, Lemma \ref{lemma_ground_closure} implies that $N[\hat{G}_{\gamma+1}]^{<\gamma}\cap V[G_{\gamma+1}]\subseteq N[\hat{G}_{\gamma+1}]$. Since $N[\hat{G}_{\gamma+1}]$ is a $\gamma$-model in $V[G_{\gamma+1}]$ and $\P'_{j(\gamma)} / \hat{G}_{\gamma+1}$ is $\gamma$-strategically closed in $N[\hat{G}_{\gamma+1}]$ we can build a filter $K\in V[G_{\gamma+1}]$ which is generic for $\P'_{j(\gamma)} / \hat{G}_{\gamma+1}$ over $N[\hat{G}_{\gamma+1}]$ such that $r'/\hat{G}_{\gamma+1} \in K$. We let $\hat{G}_{j(\gamma)}=\hat{G}_{\gamma+1}*K$. Since conditions in $\P_\gamma$ have support bounded below the critical point of $j$ we have $j[G_\gamma]\subseteq\hat{G}_{j(\gamma)}$ and thus $j$ extends to $j:M[G_\gamma]\to N[\hat{G}_{j(\gamma)}]$. By Lemma \ref{lemma_ground_closure} we have $N[\hat{G}_{j(\gamma)}]^{<\gamma}\cap V[G_{\gamma+1}]\subseteq N[\hat{G}_{j(\gamma)}]$.

Recall that $H_\gamma$ is a filter for $\Q_\gamma$ which is generic over $V[G_\gamma]$ with $r(\gamma)=(m+1,f\restrict[m+1,\omega))\in H_\gamma$ and $H_\gamma=\{m+1\}\times I_\gamma$ where $I_\gamma\subseteq\prod_{n\in[m+1,\omega)}T^1(\gamma\setminus S^\gamma_n )$. For each $n\in[m+1,\omega)$ let $C^\gamma_n=\bigcup_{h\in I_\gamma} h(n)$ and note that each $C^\gamma_n$ is in $N[\hat{G}_{j(\gamma)}]$. Moreover, by the same argument as in the proof of Lemma \ref{lemma_no_omega_wc}, $\bigcup J_\gamma$ is a $1$-club subset of $\gamma$ in $N[\hat{G}_{j(\gamma)}]$ which is disjoint from $(S^\gamma_m)^{N[\hat{G}_{j(\gamma)}]}$. Then it follows that $\gamma$ is not $(m+1)$-weakly compact in $N[\hat{G}_{j(\gamma)}]$. This implies that $F=\<C^\gamma_n\cup\{\gamma\}\st n\in[m+1,\omega)\>$ is a condition in $j(\prod_{n\in[m+1,\omega)}T^1(\gamma\setminus S^\gamma_n))$, and thus $(m+1,F)$ is a condition in $\Q'_{j(\gamma)} = (\dot{\Q}'_{j(\gamma)})_{\hat{G}_{j(\gamma)}}$ below every element of $j[H_\gamma]$. Since $N[\hat{G}_{j(\gamma)}]$ is a $\gamma$-model in $V[G_{\gamma+1}]$ it follows from Lemma \ref{lemma_diagonalization_criterion} that working in $V[G_{\gamma+1}]$, we can build a filter $\hat{H}_{j(\gamma)}$ for $\Q'_{j(\gamma)}$ generic over $N[\hat{G}_{j(\gamma)}]$ with $(m+1,F)\in \hat{H}_{j(\gamma)}$. Let $\hat{G}_{j(\gamma)+1} := \hat{G}_{j(\gamma)} * \hat{H}_{j(\gamma)}$. Since $j[H_\gamma]\subseteq\hat{H}_{j(\gamma)}$ we see that $j$ extends to $j:M[G_{\gamma+1}]\to N[\hat{G}_{j(\gamma)+1}]$. Since $\hat{H}_{j(\gamma)}\in V[G_{\gamma+1}]$ it follows from Lemma \ref{lemma_ground_closure} that $N[\hat{G}_{j(\gamma)+1}]$ is a $\gamma$-model in $V[G_{\gamma+1}]$. Since $r'\in \hat{G}_{j(\gamma)}$ we have $\gamma\in j(S)$. This contradicts $(*)$.
\end{proof}

\begin{proof}[Proof of Lemma \ref{lemma_refl_P_kappa}]
First note that $\kappa$ is weakly compact in $V^{\P_\kappa}$ by Theorem \ref{theorem_weakly_compact_ideal_after_forcing}. Hence it suffices to prove that in $V^{\P_\kappa}$ every weakly compact subset of $\kappa$ has a weakly compact proper initial segment.

Suppose $p\forces_{\P_\kappa} \text{``$\dot{S}$ is a weakly compact subset of $\kappa$''}$. We will show that the set of conditions below $p$ forcing that $\dot{S}$ has a weakly compact proper initial segment is dense below $p$. Let $p'\leq p$ be any extension of $p$. Let $T$ be the set of all $\alpha < \kappa$ such that there is $p'' \leq p'$ with $p''\forces_{\P_\kappa} \alpha \in \dot{S}$. For each $\alpha\in T$ fix a condition $p_\alpha\leq p'$ with $p_\alpha \forces_{\P_\kappa} \alpha \in \dot{S}$. Since $p\forces_{\P_\kappa} \dot{S} \subseteq T$, it follows that $T$ is weakly compact in $V$. Applying the fact that $\Refl_{\wc}(\kappa)$ holds in $V$, let $\gamma<\kappa$ be such that $T\cap\gamma$ is a weakly compact subset of $\gamma$, and $p_\alpha \in \P_\gamma$ for all $\alpha \in T \cap \gamma$. By Lemma \ref{lemma_gitik_shelah} and Theorem \ref{theorem_weakly_compact_ideal_after_forcing}, there is $\alpha\in T \cap \gamma$ such that
\[p_\alpha\forces_{\P_\gamma} \text{``$\{\beta\in T \cap \gamma \st p_\beta\in \dot{G}_\gamma\}$ is a weakly compact subset of $\gamma$''}.\]
where $\dot{G}_\gamma$ is the canonical name for a $\P_\gamma$-generic filter.
By Lemma \ref{lemma_hiroshi}, there is a condition $r\in\P_{\gamma+1}$ with $r\leq p_\alpha$ such that
\[r\forces_{\P_{\gamma+1}} \text{``$\{\beta\in T \cap \gamma \st p_\beta\in \dot{G}_\gamma\}$ is a weakly compact subset of $\gamma$''}.\]
Let us show that $r\forces_{\P_\kappa}$ ``$\dot{S} \cap \gamma$ is a weakly compact subset of $\gamma$''. Suppose $G \subseteq \P_\kappa$ is any generic filter over $V$ with $r\in G$. Let $G_\gamma := G \cap \P_\gamma$. Then $\{\beta\in T \cap \gamma \st p_\beta\in G_\gamma\}$ is a weakly compact subset of $\gamma$ in $V[G_{\gamma+1}]$, and thus also in $V[G]$ since $V[ G_\gamma ]^\gamma \cap V[G] \subseteq V[ G_\gamma ]$. But $\{ \beta \in T \cap \gamma \st p_\beta \in G_\gamma \} \subseteq \dot{S}_G \cap \gamma$. Hence $\dot{S}_G \cap \gamma$ is a weakly compact subset of $\gamma$ in $V[G]$.
\end{proof}

This completes the proof of Theorem \ref{theorem_main_theorem}.
\end{proof}

\section{Question}\label{section_questions}

Let us summarize some open questions regarding the $\Pi^1_n$-reflection principles. In this article we prove that the $\Pi^1_1$-reflection principle can hold at the least $\omega$-$\Pi^1_1$-indescribable cardinal, and thus the $\Pi^1_1$-reflection principle $\Refl_{\wc}(\kappa)$ does not imply that $\kappa$ is $(\omega+1)$-$\Pi^1_1$-indescribable.

\begin{question}
For $n>1$, can the $\Pi^1_n$-reflection principle hold at the least $\omega$-$\Pi^1_n$-indescribable cardinal? (For $n=0$ the answer is easily seen to be yes from Mekler and Shelah \cite{MR1029909}.)
\end{question}

If $\xi<\kappa^+$ is any fixed ordinal, the first author proved that if $\kappa$ is $\xi$-$\Pi^1_1$-indescribable and the $\Pi^1_1$-reflection principle $\Refl_{\wc}(\kappa)$ holds, then there is a forcing extension in which $\kappa$ gains a non-reflecting $\Pi^1_1$-indescribable set and $\kappa$ remains $\xi$-$\Pi^1_1$-indescribable. Can we add a non-reflecting weakly compact subset of $\kappa$ while preserving the great weak compactness of $\kappa$?

\begin{question}
For $n\geq 1$, if $\kappa$ is a greatly $\Pi^1_n$-indescribable cardinal and $\Refl_n(\kappa)$ holds, is there a forcing extension in which $\kappa$ gains a non-reflecting $\Pi^1_n$-indescribable set and $\kappa$ remains greatly $\Pi^1_n$-indescribable? (For $n=0$ the answer is easily seen to be yes using standard methods.)
\end{question}

Mekler and Shelah proved \cite{MR1029909} that the consistency strength of the stationary reflection principle $\Refl(\kappa)$ is strictly between the existence of a greatly Mahlo cardinal and the existence of a $\Pi^1_1$-indescribable, and is equiconsistent with the existence of a reflection cardinal; a regular uncountable cardinal $\kappa$ is a \emph{reflection cardinal} if there exists a proper normal $\kappa$-complete ideal $I$ on $\kappa$ such that $X\in I^+$ implies $\Tr_0(X)\in I^+$. Can a similar result be established for the $\Pi^1_n$-reflection principles? Bagaria and Mancilla determined the consistency strength of the existence of a \emph{$(n+1)$-stationary cardinal}, a notion which is closely related to the $\Pi^1_{n-1}$-reflection principle (see \cite{MR3416912} for a definition). Bagaria and Mancilla proved that the consistency strength of the existence of a $(n+1)$-stationary cardinal is strictly between the existence of a $n$-greatly-Mahlo cardinal and the existence of a $\Pi^1_n$-indescribable cardinal (again see \cite{MR3416912} for definitions). A similar result may be achievable for the $\Pi^1_n$-reflection principles.

\begin{question}
What is the consistency strength of the $\Pi^1_n$-reflection principle? (For $n=0$ this is answered by Mekler and Shelah \cite{MR1029909}.) 
\end{question}

A regular uncountable cardinal $\kappa$ is an \emph{$n$-reflection cardinal} if there exists a proper normal $\kappa$-complete ideal $I$ such that $A\in I^+$ implies $\Tr_n(A)\in I^+$. Is the $\Pi^1_n$-reflection principle $\Refl_n(\kappa)$ equiconsistent with the existence of an $n$-reflection cardinal and strictly between the existence of a greatly $\Pi^1_n$-indescribable cardinal and the existence of a $\Pi^1_{n+1}$-indescribable cardinal?

\begin{question}
What is the relationship between the $n$-stationarity of a cardinal $\kappa$ and the $\Pi^1_n$-reflection principle $\Refl_n(\kappa)$? In $L$, the two notions are equivalent. What about in other models? 
\end{question}


\end{document}